\documentclass[openacc]{rsproca_new}
\usepackage{graphicx}
\usepackage{tikz}
\usepackage{float}
\usepackage{enumitem}
\usepackage{amsrefs}

\newtheorem{theorem}{\bf Theorem}[section]
\newtheorem{corollary}{\bf Corollary}[section]
\newtheorem{lemma}{\bf Lemma}[section]
\newtheorem{remark}{\it Remark}[section]
\newtheorem{definition}{\bf Definition}[section]
\newcommand{\orcidauthorA}{0000-0001-7504-4444}
\newcommand{\orcidauthorB}{0000-0001-6859-7788}

\begin{document}

\title[$q$-orthogonal polynomials and the $q$-Riemann Hilbert Problem]{On a class of $q$-orthogonal polynomials and the $q$-Riemann Hilbert Problem}

\author{
Nalini Joshi$^{1}$, Tomas Lasic Latimer$^{2}$}

\address{$^{1}$School of Mathematics and Statistics F07, University of Sydney, Sydney NSW 2006, Australia\\
$^{2}$School of Mathematics and Statistics F07, University of Sydney, Sydney NSW 2006, Australia}

\subject{Applied mathematics}

\keywords{Orthogonal polynomials, Riemann Hilbert Problem, $q$-difference calculus}

\corres{Nalini Joshi\\
\email{nalini.joshi@sydney.edu.au}}

\begin{abstract}
We give an explicit solution of a $q$-Riemann Hilbert problem which arises in the theory of orthogonal polynomials, prove that it is unique, and deduce several properties. In particular we describe the asymptotic behaviour of zeroes in the limit as the degree of the polynomial approaches infinity.
\end{abstract}


\begin{fmtext}
\section{Introduction}
In this paper, we consider a Riemann-Hilbert problem (RHP) associated with a general class of $q$-orthogonal polynomials and solve it explicitly. The result is used to deduce properties of such polynomials, in particular, ladder operators, a Lax pair, a $q$-difference equation and, most importantly, the asymptotic behaviour of their zeroes in the limit as the degree of the polynomials approach infinity.

RHPs encapsulate many classical problems of physics and mathematics. Related problems (under certain conditions) are also known as Wiener-Hopf problems. Applications include boundary value problems in hydrodynamics, diffraction theory and radiative transfer theory \cite{noble1962methods}. More recently, RHPs have been used to describe solutions of partial differential equations of Korteweg-de Vries type \cites{AC91,BDT88}, exactly solvable quantum field and statistical mechanics models \cite{KBI93}, topological and two-dimensional models of quantum gravity \cite{GM90}, and are particularly noteworthy as a method for describing asymptotic behaviours in the study of ensembles of random matrices \cites{D99,D00}. 
\end{fmtext}

\maketitle
 Orthogonal polynomials are objects of longstanding interest in mathematics and physics. Classical orthogonality conditions involved continuous measures, with weight functions $w(x)$ defined on continuous domains.
The pioneering studies of Lagrange, Laguerre, Hermite and others focused on smooth weight functions $w(x)$ such that $w'(x)/w(x)$ grows no faster than a linear polynomial in $x$. This condition was extended in two cases by Shohat \cite{shohat1939}. Applications in theoretical physics have led to many more extensions studied in modern times \cites{maroni1987prolegomenes,fokas1991discrete,Magnus1999,AsscheBook}. These cases led to families of functions that are now known as \emph{semi-classical} orthogonal polynomials.

In this paper, we consider orthogonality conditions defined through a $q$-discrete measure, that is, for given $q\in\mathbb R$, $0<q<1$, orthogonality measures defined on the multiplicative lattice, $\{q^n\}_{n\in\mathbb Z}$. Given an appropriate weight function on such a lattice, we consider orthogonality conditions with respect to the Jackson integral. (We give the precise definitions in Section \ref{notation} below.) Such polynomials, called the $q$-Hahn class \cite[Section 18.27]{NIST:DLMF},  satisfy relations analogous to those satisfied by semi-classical polynomials but with the derivative operator replaced by the $q$-derivative $D_q$ (see Equation \eqref{q-div}). Known properties include ladder operators and associated pairs of compatible linear equations, called Lax pairs. We provide a new way to deduce such relations from the corresponding RHP.  

Not all properties of $q$-orthogonal polynomials are known. In this paper, we provide a new result, Theorem \ref{all zeroes}, leading to the asymptotic approximation of zeroes of $q$-orthogonal polynomials corresponding to a class of $q$-discrete measures. This leads to a generalisation of an earlier result found for polynomials orthogonal with respect to a positive discrete measure \cite[Theorem 6.1]{szeg1939orthogonal}.

\subsection{Background}\label{Background}
In 1991, Fokas, Its and Kitaev \cite{fokas1991discrete} showed how to formulate and solve RHPs for families of semi-classical orthogonal polynomials $\bigl\{n\in\mathbb N \bigm| P_{n}(x)\bigr\}$ with exponential weight functions $\exp(-V(x))$ on the real line, where $V$ is a polynomial of even degree and the coefficient of the highest-degree term is positive. Based on the RHP framework for orthogonal polynomials,  Deift \textit{et al.} \cite{Deift1999strong} developed a rigorous methodology to obtain the asymptotic behaviours of the corresponding orthogonal polynomials in the limit as the degree $n\to+\infty$. Their work was based on earlier advancements by Deift and Zhou on the steepest descent method for oscillatory RHPs  \cite{deift1993steepest}. Since then RHPs have been used to determine the asymptotic behaviour of a wide range of families of orthogonal polynomials.

Fokas {\it et al.} \cite{fokas1991discrete} also showed that the above orthogonal polynomials are related to discrete Painlev\'e equations through the coefficients of their 3-term recurrence relations. For a given measure $d\mu(x)$, monic orthogonal polynomials $\bigl\{ P_{n}(x)\bigr\}$ satisfy an orthogonality relation
\[ \int P_{n}(x)P_{m}(x) d\mu(x) = \gamma_{n}\delta_{n,m} ,\]
which defines $\gamma_{n}$. Their 3-term recurrence relation takes the form
\begin{equation}\label{recurrence}
    xP_{n}(x) = P_{n+1}(x)+b_{n}P_{n}(x)+a_{n}P_{n-1}(x) \,,
\end{equation}
where
\begin{eqnarray*}
a_{n} &=& \gamma_{n}/\gamma_{n-1} ,\\
b_{n} &=& \int xP_{n}(x)^{2} d\mu(x)/\gamma_{n}.
\end{eqnarray*}
For convenience, we take the sequence of polynomials to have initial values $P_{-1}=0$ and $P_{0} = 1$. When $V$ is quartic in $x$, actually $V(z)=-x^4/4$, it had earlier been shown \cites{shohat1939,freud1976coefficients} that $a_n$ satisfies the discrete equation
\begin{equation}
    a_{n}(a_{n+1}+a_{n}+a_{n-1})=n ,\label{dp1}
\end{equation}
Fokas {\it et al.} called it a discrete Painlev\'e equation, dPI, and proved that it has a continuum limit to the first Painlev\'e equation.
Equation \eqref{dp1} is also known as the string equation in physics due to its appearance in the Hermitian random matrix model of quantum gravity \cite{fokas1991discrete}. 

The connection between Equation \eqref{dp1} and dPI was established by considering the Lax pair satisfied by polynomials orthogonal with respect to $e^{-x^{4}/4}dx$. Semi-classical orthogonal polynomials satisfy a Lax pair given by a differential equation in the independent variable $x$ and an iteration in the degree $n$. Fokas \textit{et al.} demonstrated a connection between the Lax pair, the RHP formulation and dPI for this class of semi-classical orthogonal polynomials \cite{fokas1992isomonodromy}.

It is now known that many families of orthogonal polynomials have recurrence coefficients determined by additive-type discrete Painlev\'e equations \cites{Magnus1999,clarkson2018properties,clarkson2020generalised}. More recently, it has been shown that the recurrence coefficients of $q$-orthogonal polynomials satisfy multiplicative-type discrete Painlev\'e equations \cite{Boelen}, where now the non-autonomous term in the equation is iterated on multiplicative lattices. (For the terminology distinguishing types of discrete Painlev\'e equations, we refer to Sakai \cite{s:01}.)

Baik \textit{et al.} have described an interpolation problem \cite[Interpolation Problem 1.2]{baik2007discrete}, which can be considered as a discrete RHP for polynomials orthogonal with respect to a discrete point-wise measure. They use this setting to prove a number of asymptotic results, expanding on the approach of Deift \textit{et al.} Applications of such an approach include discrete polynomials that are, for example, orthogonal on the integer lattice. We note that its extension to the multiplicative or $q$-discrete case does not appear to lead to the results we find in this paper.

The gap in our knowledge of properties of $q$-orthogonal polynomials raises a number of questions relevant for physics, including quantum physics. Discrete $q$-Hermite I polynomials describe the dynamics of discrete quantum harmonic oscillators \cite{atakishiyev2008discrete}, little $q$-Jacobi polynomials describe the unitary representations of the quantum group $SU_{q}(2)$ \cite{masuda1991representations}, and $q$-Racah polynomials are used to describe quantum knot invariants \cite{chan2018orthogonal}. In quantum communication, perfect state transfer in linear spin chains is achieved when the spin chain data is related to the Jacobi matrix of $q$-Krawtchouk polynomials \cite{jafarov2010quantum}. We note that $q$-orthogonal polynomials also appear in other fields, for example the birth and death process \cite{sasaki2009exactly}. 

This paper demonstrates the connection between $q$-orthogonal polynomials, discrete $q$-Painlev\'e equations \cite{Boelen} and $q$-RHPs. Birkhoff \cite{Birkhoff1913} and Carmichael \cite{Carmichael} pioneered the study of $q$-RHPs, however their work was initially restricted to to fuchsian systems. Later, Adams extended their work beyond fuchsian systems \cite{Adams}. More recently, the application of Galois theory has been used to derive an analytic approach to singular regular linear $q$-difference systems \cite{sauloy2003galois}. Isomonodromy conditions have led to RHPs corresponding to $q$-difference Painlev\'e equations $q$-$P_{\rm{IV}}$ \cite{joshi2021riemann} and $q$-$P_{\rm{VI}}$ \cite{Yousuke2020} .

\subsection{Notation}\label{notation}
For completeness, we recall some well known definitions and notations from the calculus of $q$-differences. These definitions can be found in \cite{Ernst2012}. Throughout the paper we will assume $q \in \mathbb{R}$ and $0<q<1$.

\begin{definition}
We define the $q$-derivative $D_{q}$, Pochhammer symbol $(a;q)_{\infty}$, and Jackson integrals as follows.
\begin{enumerate}
\item The $q$-derivative, $D_{q}$, of $f(x)$ is defined as
\begin{equation}\label{q-div}
    D_{q}f(x) = \frac{f(x)-f(qx)}{x(1-q)} \,.
\end{equation}
Note that 
\[ D_{q}x^{n} = \frac{1-q^{n}}{1-q}x^{n-1} \,. \]

\item The Pochhammer symbol $(a;q)_{\infty}$ is defined as
\begin{equation*}
    (a;q)_{\infty} = \prod_{j=0}^{\infty}(1-aq^{j}) \,.
\end{equation*}

\item There are two Jackson integrals in the literature. The Jackson integral of $f(x)$ from 0 to 1 is defined as
\begin{subequations}\label{Jackson}
\begin{equation}
    \int_{0}^{1}f(x)d_{q}x = (1-q)\sum_{k=0}^{\infty} f(q^{k})q^{k} \,.
\end{equation}
Similarly, the Jackson integral of $f(x)$ from -1 to 1 is defined as
\begin{equation}
    \int_{-1}^{1}f(x)d_{q}x = (1-q)\sum_{k=0}^{\infty} (f(q^{k})+f(-q^{k}))q^{k} \,.
\end{equation}
\end{subequations}
\end{enumerate}
\end{definition}

Note that the results in this paper are applicable to both types of Jackson integrals. 
Typical examples of $q$-orthogonal polynomials with a Jackson integral from -1 to 1 include: discrete $q$-Hermite I polynomials (with weight $w(x) = (q^{2}x^{2};q^{2})_{\infty}$) \cite[Section 18.27]{NIST:DLMF} and $q$-Freud polynomials $\left(w(x) = x^{\kappa}(q^{4}x^{4};q^{4})_{\infty}\right)$ \cite{Boelen} . Examples of $q$-orthogonal polynomials with a Jackson integral from 0 to 1 include: little $q$-Jacobi $\left(w(x) = x^{\kappa}(qx;q)_{\infty}/(bqx;q)_{\infty}\right)$ and little $q$-Laguerre polynomials $\left(w(x) = x^{\kappa}(qx;q)_{\infty}\right)$ \cite{van2020zero} .

\subsection{Main results}
Our main results hold under certain properties defined below.
\begin{definition}\label{admissable}
A positively oriented Jordan curve $\Gamma$ in $\mathbb C$ with interior $\mathcal D_-\subset\mathbb C$ and exterior $\mathcal D_+\subset\mathbb C$ is called {\em appropriate} if 
\[
\pm q^m\in \begin{cases}
&\mathcal D_- \quad {\rm if}\, m\ge 0,\\
&\mathcal D_+ \quad {\rm if}\, m< 0.
\end{cases}
\]
We call $w$ {\em admissible} if it is a function analytic in $\mathcal D_-$ and $w(z)\in \mathbb{R}$ for $z\in \mathbb{R}$. Furthermore, we call $w$ {\em strictly positive admissible} if $w(z)>0$ at $z = \pm q^{m}$.

\end{definition}
An example of an appropriate curve is drawn in Figure \ref{RHP-figure}.
\begin{remark}
In Definition \ref{admissable} we describe the conditions needed for $q$-orthogonal polynomials with Jackson integral from -1 to 1. For a Jackson integral from 0 to 1 the same definition applies with $\pm q^{k} \rightarrow q^{k}$.
\end{remark}

\begin{definition}[$q$-Riemann-Hilbert problem]\label{qRHP}
Let $\Gamma$ be an appropriate curve (see Definition \ref{admissable}) with interior $\mathcal D_-$ and exterior $\mathcal D_+$. Let $w$ be an admissible function in $\mathcal D_-$ and define 
\[
h(s) = \sum_{k=-\infty}^{\infty} \frac{2sq^{k}}{s^{2}-q^{2k}}  = \sum_{k=-\infty}^{\infty} \left( \frac{q^{k}}{s-q^{k}} + \frac{q^{k}}{s+q^{k}} \right).
\]
A $2\times 2$ complex matrix function $Y(z)$, $z\in\mathbb C$, is called a \emph{solution of the $q$-Riemann-Hilbert problem} if it satisfies the following conditions:
    \begin{enumerate}[label={{\rm (\roman *)}}]
    \item $Y(z)$ is analytic on $\mathbb{C}\setminus \Gamma$.
\item $Y(z)$ has continuous boundary values $Y_-(s)$ and $Y_+(s)$ as $z$ approaches $s\in\Gamma$ from $\mathcal D_-$ and $\mathcal D_+$ respectively, where
\begin{gather} \label{b}
 Y_+(s) =
  Y_-(s)
  \begin{pmatrix}
   1 &
   w(s)h(s) \\
   0 &
   1 
   \end{pmatrix}, \; s\in \Gamma .
\end{gather}
\item $Y(z)$ satisfies
\begin{gather} \label{c}
Y(z)\begin{pmatrix}
   z^{-n} &
   0 \\
   0 &
   z^{n} 
   \end{pmatrix}
 =
I + O\left( \frac{1}{|z|} \right), \text{ as }\ |z| \rightarrow \infty.
\end{gather}
\end{enumerate}
\end{definition}
\begin{remark}
    Note that the function 
    \[
    h(z) = \sum_{k=-\infty}^{\infty} \frac{2zq^{k}}{z^{2}-q^{2k}}  = \sum_{k=-\infty}^{\infty} \left( \frac{q^{k}}{z-q^{k}} + \frac{q^{k}}{z+q^{k}} \right),
    \]
    is meromorphic and is not continuous at $z=0$ which is an accumulation point of the poles of $h(z)$. Furthermore, note that $h(qz) = h(z)$. Functions of this nature play an important role in $q$-Riemann Hilbert theory where the elements of the transition matrix ($(Y_-)^{-1}Y_+$ in Equation \eqref{b}) can be related to Weierstrass elliptic functions, see \cite[Section 19]{Birkhoff1913}.
\end{remark}
The main results of this paper are given by the following theorems.
\begin{theorem}[$q$-RHP] \label{RHPtheorem}
Let $P_{n}(z)$ be the monic $n^{th}$-degree orthogonal polynomial with respect to the admissible measure $w(x)d_{q}x$, i.e., 
\[ \int^{1}_{-1} P_{n}P_{m}w(x)d_{q}x = \gamma_{n}\delta_{n,m} .\]
Then $Y(z)$ defined by \\
\begin{gather} \label{RHP sol}
Y(z) 
=
\begin{bmatrix}
   P_{n}(z) &
   \oint_{\Gamma}\frac{P_{n}(s)w(s)h(s)}{2\pi i (z-s)}ds \\
   \gamma_{n-1}^{-1} P_{n-1}(z) &
   \oint_{\Gamma}\frac{P_{n-1}(s)w(s)h(s)}{ 2\pi i (z-s)\gamma_{n-1}}ds
   \end{bmatrix}
\end{gather}
is a unique solution of the Riemann-Hilbert problem given in Definition \eqref{qRHP}.
\end{theorem}

\begin{theorem}\label{all zeroes}
    Let $P_{n}(x)$ be the monic $n^{\rm th}$-degree $q$-orthogonal polynomial with respect to the admissible weight $w(x)$ on $[-1,0)\cup(0,1]$, i.e., 
\[ \int^{1}_{-1} P_{n}P_{m}w(x)d_{q}x = \gamma_{n}\delta_{n,m} .\]
 Denote the zeroes of $P_n(x)$ by $\{x_{i,n}\}_{i=1}^n$. Suppose that the recurrence coefficients defined in Equation \eqref{recurrence} decay sufficiently fast such that
 \[ \sum_{n=0}^{\infty}|a_n|+|b_n| < \infty .\]
 Then , for any $k \in \mathbb{N}$ there exists $1\le i\le n$ such that $x_{i,n}\rightarrow q^{k}$ as $n \rightarrow \infty$.
\end{theorem}
\begin{remark}
Theorem \ref{all zeroes} holds for $q$-orthogonal polynomials which may have non-positive admissible weights.
\end{remark}
\begin{remark}\label{r2.9}
Lemma \ref{2.3} demonstrates that $q$-orthogonal polynomials with strictly positive admissible weights satisfy Theorem \ref{all zeroes}. Furthermore, we will show in Lemma \ref{convergence} that there exists an upper bound on the rate of convergence proportional to $M_{1}z^{-2}+M_{2}z^{-1}$.
\end{remark}

 We use these results to prove a number of properties satisfied by $q$-orthogonal polynomials. From Theorem \ref{RHPtheorem}, we deduce the ladder operators as well as the Lax pair satisfied by $q$-orthogonal polynomials. In particular, we show that the solution of the $q$-RHP solves a $q$-difference equation in the independent variable $z$.

\subsection{Outline}
The paper is structured as follows. In Section \ref{method} we prove Theorems \ref{RHPtheorem} and \ref{all zeroes}. We begin by solving the $q$-RHP in Section \ref{RHP section}. Using the $q$-RHP formalism, in Sections \ref{Lax Pair section} and \ref{Ladder operator section} we deduce the Lax pair and ladder operators respectively. In Section \ref{Zero section} we describe the asymptotic location of zeroes using the $q$-RHP. We then conclude with Section \ref{conc sec} where we discuss future work and summarise our results.

\begin{figure}
\centering
\begin{tikzpicture}
\draw[latex-latex] (-5,0) -- (5,0);
\draw[latex-latex] (0,-4) -- (0,4);
\draw (0,4) node[left] {\large $\Im(z)$};
\draw (5,0) node[above] {\large $\Re(z)$};

\foreach \i in {-0.0625,-0.125,-0.25,-0.5, -1, -2,-4,0,0.0625,0.125,0.25,0.5, 1, 2,4}
\fill[black] (\i,0) circle (0.4 mm);

\draw (-2,0) node[above] {$-1$};
\draw (2,0) node[above] {$1$};
\draw (-4,0) node[above] {\small $-q^{-1}$};
\draw (4,0) node[above] {\small $q^{-1}$};
\draw (-1,0) node[above] {\small $-q$};
\draw (1,0) node[above] {\small $\,q$};

\draw[thick,dashed] (0,0) circle (3cm);
\draw (1,1) node[above] {\large $Y_{-}$};
\draw (2.5,2.5) node[above] {\large $Y_{+}$};
\draw (-2.4,2.4) node[above] {\large $\Gamma$};

\end{tikzpicture}  
\caption{Example of an appropriate curve $\Gamma$ defined by Definition \ref{admissable}.}
\label{RHP-figure}
\end{figure}
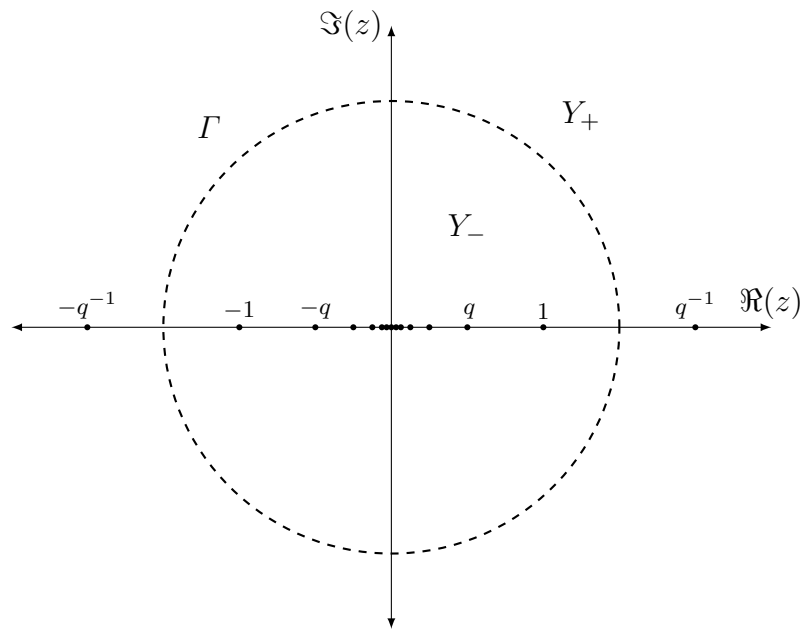

\section{Proofs of main results}\label{method}
In this section, we prove Theorems \ref{RHPtheorem} and \ref{all zeroes} and provide additional results that lead to ladder operators and a Lax pair. As remarked earlier, we focus on the case of the Jackson integral from -1 to 1 (see Equation \eqref{Jackson}). The results for the case of a Jackson integral from 0 to 1 follow using similar arguments.

\subsection{Proof of Theorem \ref{RHPtheorem}}\label{RHP section}

The proof follows along the same lines as in the case of the RHP for semi-classical orthogonal polynomials given by Fokas \textit{et al.} (see \cite[Section 3]{Deift1999strong} for a detailed proof). When moving from the semi-classical case to $q$-orthogonal polynomials a key step relies on the use of the Cauchy residue theorem to express a contour integral as a discrete sum. 
\begin{proof}
We show that the second row of $Y$ must be given by Equation \eqref{RHP sol}. A similar argument can be carried out for the first  row.   \\ \\
It follows from the asymptotic condition, Equation \eqref{c}, that the $(1,1)$ entry of $Y$ must have leading order $z^{n}$ as $z\rightarrow \infty$. As $Y_{(1,1)}$ is analytic and its jump condition, Equation \eqref{b}, is given by the identity we immediately conclude that $Y_{(1,1)}$ is a monic polynomial of degree $n$. Similarly, it follows that $Y_{(2,1)}$ is a polynomial of degree at most $n-1$. We denote $Y_{(2,1)}$ by $Q_{n-1}$.  \\ \\
Consider the bottom right entry of Equation \eqref{RHP sol}. By the jump condition, Equation \eqref{b}, we have
\begin{equation}\label{1d jump}
Y_{(2,2)}^{+} = Q_{n-1}(s)w(s)h(s) + Y_{(2,2)}^{-} \,. \end{equation}
This scalar equation is solved by the Cauchy transform
\begin{equation}\label{y22 solved}
    Y_{(2,2)}(z) = \frac{1}{2\pi i}\oint_{\Gamma}\frac{Q_{n-1}(s)w(s)h(s)}{z-s}ds \,,
\end{equation} 
which is analytic in $\mathbb{C}\setminus \Gamma$ and satisfies Equation \eqref{1d jump}. \\ \\
The only step remaining is to prove the asymptotic condition, Equation \eqref{c}, for $Y_{(2,2)}$. Substituting our expression for $h(s)$ into Equation \eqref{y22 solved}, we find
\[ Y_{(2,2)}(z) = \frac{1}{2\pi i}\oint_{\Gamma}\sum_{k=-\infty}^{\infty} \left( \frac{q^{k}}{s-q^{k}}\frac{Q_{n-1}(s)w(s)}{z-s} + \frac{q^{k}}{s+q^{k}}\frac{Q_{n-1}(s)w(s)}{z-s} \right) ds \,. \]
which by Cauchy's integral formula, for $z\in \text{ext}(\Gamma)$, becomes
\begin{eqnarray*}
Y_{(2,2)}(z) &=& \sum_{k=0}^{\infty} \left( q^{k}\frac{Q_{n-1}(q^{k})w(q^{k})}{z-q^{k}} + q^{k}\frac{Q_{n-1}(-q^{k})w(-q^{k})}{z+q^{k}} \right) \,, \\
&=& \int^{1}_{-1} \frac{Q_{n-1}(x)w(x)}{z-x} d_{q}x \,,
\end{eqnarray*}
where the sum to infinity is well defined on $\Gamma$, as $h(s)$ converges as $k\rightarrow \infty$, and the Jackson integral of an analytic function is well defined. Note that only the points inside the jump curve $\Gamma$ contribute to the integral. \\ \\
Using the geometric series with remainder
\begin{equation}\label{exp}
    \frac{1}{z-x} = \sum^{l}_{k=0}\left( \frac{x^{k}}{z^{k+1}} \right) + \frac{x^{l+1}}{z^{l+1}(z-x)}\,, \quad \text{for}\,x\neq z \,,
\end{equation}
we find
\begin{eqnarray*}
Y_{(2,2)}(z) &=& \int^{1}_{-1} \frac{Q_{n-1}(x)w(x)x^{n}}{z^{n}(z-x)} d_{q}x \\
&& + \sum^{n-1}_{k=0} \frac{1}{z^{k+1}}\int^{1}_{-1} Q_{n-1}(x)w(x)x^{k} d_{q}x  \,.
\end{eqnarray*} 
Note that the asymptotic condition, Equation \eqref{c}, holds when the last term on the RHS is zero for $k = 0,1,2,...,n-2$. This is true iff 
\[ \int^{1}_{-1} Q_{n-1}(x)w(x)x^{k} d_{q}x  = 0\,, \quad \text{for} \, k\leq n-2 \,, \]
which is satisfied when $Q_{n-1}$ is an orthogonal polynomial of degree $n-1$ on the $q$-lattice with respect to the weight $w(x)$. We conclude the solution of $Y_{(2,2)}(z)$ is given by
\[
    Y_{(2,2)}(z) = \left\{\begin{array}{lr}
       \frac{1}{\gamma_{n-1}}\int^{1}_{-1} \frac{P_{n-1}(x)w(x)}{z-x} d_{q}x -\frac{1}{\gamma_{n-1}} P_{n-1}(z)w(z)h(z), & \text{for } z \in \text{int}(\Gamma),\\
        \frac{1}{\gamma_{n-1}}\int^{1}_{-1} \frac{P_{n-1}(x)w(x)}{z-x} d_{q}x, & \text{for } z \in \text{ext}(\Gamma).
        \end{array}\right.
\]
After appropriate scaling, and repeating the same arguments for the first row, it follows that Equation \eqref{RHP sol} is a solution of the $q$-RHP given by Definition \ref{qRHP}.\\ \\
Uniqueness of this solution follows from consideration of the determinant. Observe that the jump matrix $J = Y_{-}^{-1}Y_{+}$ satisfies $\text{det}(J) = 1$. It immediately follows that $\text{det}(Y^{+}) = \text{det}(Y^{-})$ on $\Gamma$. Thus, $\text{det}(Y)$ is an entire function. By the asymptotic condition, Equation \eqref{c}, $\text{det}(Y) \rightarrow 1$, hence by Liouville's theorem $\text{det}(Y) = 1$. As $\text{det}(Y) = 1$ we can deduce $Y^{-1}$ exists and is analytic in $\mathbb{C}/\Gamma$. Suppose that there exists a second solution to the $q$-RHP, let us denote this solution by $\widehat{Y}$. If we define $M = \widehat{Y}Y^{-1}$, it follows that the jump conditions effectively cancel and $M_{+} = M_{-}$. Thus, $M$ is entire and $M \rightarrow I$ as $z\rightarrow \infty$. Hence, by Liouville's theorem $M = I$. We conclude that $\widehat{Y} = Y$ and there is a single unique solution to the $q$-RHP.
\end{proof}

\subsection{Lax Pair}\label{Lax Pair section}
In this section, we deduce a $q$-difference Lax pair for $q$-orthogonal polynomials from the $q$-RHP. The Lax pair for a solution $Y_{n}$ to the $q$-RHP can be written as:
\begin{eqnarray}
\Psi_{n+1}(z) &=& C_{n}(z)\Psi_{n}(z) \,, \label{nlax1}\\
\Psi_{n}(qz) &=& D_{n}(z)\Psi_{n}(z)\,, \label{qlax1}
\end{eqnarray}
where,
\begin{gather} \nonumber
\Psi_{n}(z) = 
 Y_{n}(z) \begin{bmatrix}
   w(z) &
   0 \\
   0 &
   1
   \end{bmatrix} \,.
\end{gather}
and, $C_{n}$ and $D_{n}$ both have polynomial entries. We first look at an iteration in $n$. The proof mirrors that of the semi-classical case.
\begin{corollary}\label{2.1}
Let $Y_{n}(z)$ be a solution to the $q$-RHP defined in Theorem \ref{RHPtheorem}. Then $Y_{n}$ satisfies Equation \eqref{nlax1}, where
\begin{gather} \nonumber
C_{n}(z) =
\begin{bmatrix}
z+c_{n,1} &
c_{n,2} \\
c_{n,3} &
0
\end{bmatrix} \,,
\end{gather}
and $c_{n,1}$, $c_{n,2}$ and $c_{n,3}$ are constants dependent on $n$.
\end{corollary}

\begin{proof}
We observe that the jump condition for the $q$-RHP does not depend on $n$. Thus $Y_{n+1}Y_{n}^{-1}$ is entire. By considering the asymptotic condition, Equation \eqref{c}, and using Liouville's Theorem we find
\begin{gather}\nonumber
Y_{n+1}Y_{n}^{-1} = 
\begin{bmatrix}
z + c_{n,1} &
c_{n,2} \\
c_{n,3} &
0
\end{bmatrix} \,,
\end{gather}
where $c_{n,1}$, $c_{n,2}$ and $c_{n,3}$ are constants which can be explicitly derived using the recurrence relation Equation \eqref{recurrence}.
\end{proof}
We now show that if the weight $w(z)$ satisfies a $q$-difference equation of a certain form then $Y_{n}$ satisfies Equation \eqref{qlax1}. The condition on $w(z)$ is the $q$ analogue of the Pearson relation satisfied by semi-classical weights \cite[Section 1.1.1.]{AsscheBook}.
\begin{corollary}\label{qlax lemma}
Let $Y_{n}(z)$ be a solution to the $q$-RHP defined in Definition \ref{qRHP} with admissible weight $w(z)$ such that $w(\pm q^{-1}) = 0$. Furthermore, suppose that $w(z)$ is admissible for the Jordan curve $\widetilde{\Gamma}(z) = \Gamma(q^{-1}z)$ (i.e. it is the initial curve scaled by $q^{-1}$). Define,
\begin{gather} \nonumber
\Psi_{n}(z) = 
 Y_{n}(z) \begin{bmatrix}
   w(z) &
   0 \\
   0 &
   1
   \end{bmatrix} \,.
\end{gather}
If $w(qz)/w(z)$ is a polynomial, $\tau(z)$, of degree $m$. Then
\begin{gather} \nonumber
\Psi_{n}(qz) = 
 D_{n}(z)\Psi_{n}(z)  \,,
\end{gather}
where $D_{n}(z)$ has polynomial entries with maximum degree $m$.
\end{corollary}

\begin{proof}
Define,
\begin{gather} \nonumber
D_{n} = \Psi_{n}(qz)\Psi_{n}^{-1}(z)  \,.
\end{gather}
We note that there may be issues due to zeroes or poles of the weight function $w$. However we find,
\begin{align} \nonumber
D_{n} &= Y_{n}(qz) \begin{bmatrix}
   w(qz)/w(z) &
   0 \\
   0 &
   1
   \end{bmatrix} Y_{n}(z)^{-1}
   \,, \\
   &= Y_{n}(qz) \begin{bmatrix}
   \tau(z) &
   0 \\
   0 &
   1
   \end{bmatrix} Y_{n}(z)^{-1}
   \,.
\end{align}
Therefore, $D_{n}$ is well defined by taking the limit near poles and zeroes of $w(z)$. $\Psi(z)$ satisfies the jump condition,
\begin{gather}\nonumber
 \lim_{z \rightarrow \widetilde{\Gamma}_{+}} \Psi(z)
 =
  \lim_{z \rightarrow \widetilde{\Gamma}_{-}} \Psi(z)
  \begin{bmatrix}
   1 &
   h(s) \\
   0 &
   1 
   \end{bmatrix}, \; s\in \widetilde{\Gamma}.
\end{gather}
Similarly $\Psi(qz)$ satisfies the jump condition,
\begin{gather}\nonumber
 \lim_{z\rightarrow \widetilde{\Gamma}_{+}} \Psi_{n}(qz)
 =
  \lim_{z \rightarrow \widetilde{\Gamma}_{-}} \Psi_{n}(qz)
  \begin{bmatrix}
   1 &
   h(s) \\
   0 &
   1 
   \end{bmatrix}, \; s\in \widetilde{\Gamma},
\end{gather}
as $\Gamma(z) = \widetilde{\Gamma}(qz)$ and $h(qs) = h(s)$. Observing that the jump conditions are equivalent, we can write the RHP for $D_{n}$ as:
\begin{enumerate}
    \item $D_{n}$ is analytic in $\mathbb{C}\setminus \widetilde{\Gamma}$.
    \item $D_{n}^{+}(s) = D_{n}^{-}(s)$ for $s \in \widetilde{\Gamma}$.
    \item $D_{n}$ is of degree $z^{m}$ as $z$ approaches infinity.
\end{enumerate}
Thus $D_{n}$ has polynomial entries of degree at most $m$. By definition of $D_{n}$ we have
\begin{align*}
\Psi_{n}(qz) &= D_{n}(z)\Psi_{n}(z) ,\\
Y_{n}(qz) &= D_{n}(z)Y_{n}(z)\begin{bmatrix}
\tau(z)^{-1} &
0 \\
0 &
1 
\end{bmatrix}.
\end{align*}
Corollary \ref{qlax lemma} follows, as well as Equation \eqref{qlax1}.
\end{proof}
\begin{remark}
    Repeating the above arguments we find that Corollary \ref{qlax lemma} also applies to weights where $w(q^{-1}z)/w(z)$ is a polynomial. An example is discrete $q$-Hermite I polynomials whose weight satisfies $w(q^{-1}z)/w(z) = (1-z^{2})$. Discrete $q$-Hermite I polynomials satisfy the $q$-difference equation
    \begin{gather*} 
    \begin{bmatrix}
       P_{n}(qz) \\
       P_{n-1}(qz)
   \end{bmatrix}
   = \begin{bmatrix}
       1 &
       z(q^{n}-1) \\
       zq^{2-n} &
       1-z^{2}q^{2-n}
   \end{bmatrix}
   \begin{bmatrix}
       P_{n}(z) \\
       P_{n-1}(z)
   \end{bmatrix}.
\end{gather*}
    
\end{remark}

\subsection{Ladder operators}\label{Ladder operator section}
In this section we recover ladder operators for $q$-orthogonal polynomials \cite{chen2008ladder} from the $q$-RHP. The process is similar to that of semi-classical orthogonal polynomials, the derivation of which can be found here \cite[Theorem 4.1]{AsscheBook}. To avoid onerous calculations, we restrict ourselves to determining the lowering ladder operator for $q$-orthogonal polynomials.

\begin{corollary}
Let $\{P_{n}\}_{n=0}^{\infty}$ be a sequence of monic $q$-orthogonal polynomials which satisfy a $q$-RHP with admissible weight $w(z)$ such that $w(\pm q^{-1}) = 0$. Then
\begin{equation*}
    D_{q}P_{n}(z) = A_{n}(z)P_{n-1}(z) - B_{n}(z)P_{n}(z) \,,
\end{equation*}
where
\begin{eqnarray*}
    A_{n}(z) &=& \frac{1}{\gamma_{n-1}}\int^{1}_{-1}\frac{u(y)}{qz-y}P_{n}(y)P_{n}(y/q)w(y)d_{q}y \,,\\
    B_{n}(z) &=& \frac{1}{\gamma_{n-1}}\int^{1}_{-1}\frac{u(y)}{qz-y}P_{n}(y)P_{n-1}(y/q)w(y)d_{q}y \,,
\end{eqnarray*}
and
\begin{equation*}
    u(z) = \frac{D_{q^{-1}}w(z)}{w(z)}.
\end{equation*}
\end{corollary}
\begin{proof}
Define $R=D_{q}(Y)Y^{-1}$. Then it follows that
\begin{equation}\label{R11}
   \begin{aligned} -\gamma_{n-1}R_{(1,1)} &= D_{q}(P_{n}(z))\int_{-1}^{1}\frac{P_{n-1}(y)w(y)}{y-z}d_{q}y\\ &\qquad - P_{n-1}(z)D_{q}\left(\int_{-1}^{1}\frac{P_{n}(y)w(y)}{y-z}d_{q}y\right)  \,.
   \end{aligned}
\end{equation}
Using the orthogonality of $P_{n-1}$ we can apply the geometric series, Equation \eqref{exp}, to write
\begin{equation}\label{the above equation}
    D_{q}(P_{n}(z))\int_{-1}^{1}\frac{P_{n-1}(y)w(y)}{y-z}d_{q}y = \int_{-1}^{1}\frac{D_{q}(P_{n}(y))P_{n-1}(y)w(y)}{y-z}d_{q}y \,.
\end{equation} 
By substituting in the definition of the $q$-derivative into the RHS of Equation \eqref{the above equation} we find
\[ D_{q}(P_{n}(z))\int_{-1}^{1}\frac{P_{n-1}(y)w(y)}{y-z}d_{q}y = \int_{-1}^{1}\frac{(P_{n}(qy)-P_{n}(y))P_{n-1}(y)w(y)}{y(q-1)(y-z)}d_{q}y \,.\]
Similarly, we can substitute in the definition of the $q$-derivative to write the last term in Equation \eqref{R11} as,
\begin{multline}\label{multi}
 P_{n-1}(z)\int_{-1}^{1}P_{n}(y)w(y)\left(\frac{1}{z(q-1)(y-qz)} - \frac{1}{z(q-1)(y-z)} \right) d_{q}y \\
 = P_{n-1}(z)\int_{-1}^{1}P_{n}(y)w(y)\left(\frac{q}{y(q-1)(y-qz)} - \frac{1}{y(q-1)(y-z)} \right) d_{q}y\,,   
\end{multline}
where the equality can be verified by direct calculation. Using the orthogonality of $P_{n}$ and the geometric series Equation \eqref{exp} we can write the RHS of Equation \eqref{multi} as 
\begin{multline*}
P_{n-1}(z)\int_{-1}^{1}P_{n}(y)w(y)\left(\frac{q}{y(q-1)(y-qz)} - \frac{1}{y(q-1)(y-z)} \right) d_{q}y \\
 = \int_{-1}^{1}P_{n}(y)w(y)\left(\frac{qP_{n-1}(y/q)}{y(q-1)(y-qz)} - \frac{P_{n-1}(y)}{y(q-1)(y-z)} \right) d_{q}y\,.
\end{multline*}
Equation \eqref{R11} can now be written as,
\begin{equation}\label{exp2}
    -\gamma_{n-1}R_{1,1} = \int_{-1}^{1}\frac{P_{n}(qy)P_{n-1}(y)w(y)}{y(q-1)(y-z)}d_{q}y - q\int_{-1}^{1}\frac{P_{n}(y)P_{n-1}(y/q)w(y)}{y(q-1)(y-qz)}d_{q}y  \,.
\end{equation}
From the definition of the Jackson integral (see Equation \eqref{Jackson}) we can immediately determine the identity,
\begin{equation}\nonumber
    q\int_{-1}^{1}f(qx)d_{q}x =  \int_{-1}^{1}f(x)d_{q}x - (f(1)+f(-1))\,. 
\end{equation} 
We can apply the above identity to Equation \eqref{exp2} to give us,
\begin{eqnarray*}
    -\gamma_{n-1}R_{1,1} &=& q\int_{-1}^{1}\frac{P_{n}(y)P_{n-1}(y/q)w(y/q)}{y(q-1)(y-qz)}d_{q}y - q\int_{-1}^{1}\frac{P_{n}(y)P_{n-1}(y/q)w(y)}{y(q-1)(y-qz)}d_{q}y  \,, \\
    &=& \int_{-1}^{1}\frac{P_{n}(y)P_{n-1}(y/q)Dq^{-1}(w(y))}{(qz-y)}d_{q}y .
\end{eqnarray*}
By definition, $D_{q}(Y) = RY$. Hence, $D_{q}Y_{(1,1)} = R_{(1,1)}Y_{(1,1)} + R_{(1,2)}Y_{(1,2)}$. Thus, we have just proven the form of $B_{n}$. The proof of $A_{n}$ follows a similar argument; one finds $R_{(1,2)}$ using the same method as $R_{(1,1)}$ which gives the desired result.
\end{proof}

\subsection{Proof of Theorem \ref{all zeroes}}\label{Zero section}
We know from Corollary \ref{2.1} that the two columns of the solution $Y_{n}$ to the $q$-RHP both satisfy the same three term recurrence relation, where $a_{n}$ and $b_{n}$ are given by Equation \ref{recurrence}. We will now show that they converge to analytic functions if the recurrence coefficients $a_{n}$ and $b_{n}$ decay sufficiently fast.   
\begin{lemma}\label{convergence}
 Let $y_{n}(z)$ be determined by the recurrence relation
 \[ y_{n+2}(z) -z y_{n+1}(z) + b_{n+1}y_{n+1}(z)+a_{n+1}y_{n}(z) = 0 \,, \]
 where $|z|>0$. If 
\begin{equation}\label{sufficient decay}
    \sum_{n=0}^{\infty} |a_{n}| +|b_{n}| < \infty ,
\end{equation} 
then
\begin{equation}\label{yn limit}
\lim_{n\to\infty}y_{n}(z)z^{-n}\rightarrow f(z),
\end{equation}
where convergence is uniform in any compact domain not containing $z=0$ or poles of $y_{n}$.
\end{lemma}
\begin{proof}
Substituting in $y_{n} = z^{n}u_{n}$ the recurrence relation becomes:
\begin{equation}\label{cauchy}
    u_{n+2} - u_{n+1} +\frac{b_{n+1}}{z}u_{n+1}+ \frac{a_{n+1}}{z^{2}}u_{n} = 0 \,. 
\end{equation}
Taking the norm we have
\[ |u_{n+2}| \leq \left( 1+\left| \frac{a_{n+1}}{z^{2}}\right| + \left| \frac{b_{n+1}}{z} \right| \right)\text{max}(|u_{n+1}|,|u_{n}|) \,. \] 
Using the above relation we can write
\[ \text{max}(|u_{n+2}|,|u_{n+1}|) \leq \left( 1+\left| \frac{a_{n+1}}{z^{2}}\right| + \left| \frac{b_{n+1}}{z} \right| \right)\text{max}(|u_{n+1}|,|u_{n}|) \,. \] 
Thus, it follows
\[ |u_{n}| \leq |u_{0}| \prod_{k=0}^{n}\left( 1+\left| \frac{a_{k+1}}{z^{2}}\right| + \left| \frac{b_{k+1}}{z}\right| \right) \,, \]
by applying a telescopic product. If $a_{n}$ and $b_{n}$ satisfy Equation \eqref{sufficient decay} then the RHS is bounded. Thus, $|u_{n}|$ is bounded in any compact domain not containing $z=0$. Hence, by Equation \eqref{cauchy} $u_{n}$ is a Cauchy sequence which converges as $n$ goes to infinity. Moreover, $u_{n}$ converges uniformly in any compact domain not containing $z=0$.
\end{proof}

We now prove Theorem \ref{all zeroes}. First, observe that the $q$-RHP jump, Equation \eqref{b}, can be written in one-dimension as
\begin{equation}\label{1d cauchy}
    P_{n}hw + \pi_{-}(P_{n}) = \pi_{+}(P_{n}) \,,
\end{equation} 
where $P_{n} = Y^{\pm}_{(1,1)}$, $\pi_{-} = Y^{-}_{(1,2)}$ and $\pi_{+} = Y^{+}_{(1,2)}$. Motivated by the form of Equation \eqref{yn limit}, we multiply through by a common factor $z^{-n}$ to give
\[ z^{-n}P_{n}hw + z^{-n}\pi_{-}(P_{n}) = z^{-n}\pi_{+}(P_{n}) \,.\]
Applying the results of Lemma \ref{convergence}, in the limit $n$ goes to infinity the above becomes
\[ f^{(0)}(z)hw + f^{(1)}(z) = f^{(2)}(z) \,,\]
where $f^{(0)}$ and $f^{(1)}$ are analytic everywhere except $z=0$ and $f^{(2)}$ is analytic everywhere except  $z=0$ and $z=\pm q^{k}$ for $k \in \mathbb{Z}$. The analyticity of $f^{(i)}$ follows from Morera's theorem as convergence is uniform in any compact domain D not containing $0$ or the poles of $Y_{(1,2)}$. \\ \\
Furthermore, we can deduce that $f^{(2)}(z)$ is the zero function. We know that $f^{(2)}(z)$ is analytic at infinity. We also know $z^{-n}\pi_{+}(P_{n})$ decays as $z^{-2n}\gamma_{n}^{-1}$ as $z\rightarrow \infty$. Hence, $z^{m}f^{(2)}(z)$ goes to zero for any $m$. By Liouville's theorem we conclude $f^{(2)}(z)$ is the zero function.
Thus,
\[ f^{(0)}(z)hw + f^{(1)}(z) = 0 \,,\]
If $f^{(0)}(z)$ is non-zero at the poles of $h(z)$ then by the above equation $f^{(1)}(z)$ must have a pole there. But $f^{(1)}(z)$ is analytic for $z\neq0$ which is a contradiction. Hence, $f^{(0)}(z)$ must have zeroes at the poles of $h(z)$, which are at $z=\pm q^{k}$ for $k \in \mathbb{Z}$. \\ \\
By definition 
\[ f^{(0)}(z) = \lim_{n\to\infty}z^{-n}P_{n}(z) = \lim_{n\to\infty}\prod_{i=1}^{n}\left(1-\frac{z_{i,n}}{z}\right) \,, \]
where $z_{i,n}$ are the ordered zeroes of $P_{n}$. Thus, the zeroes of $f^{(0)}(z)$ are the zeroes of $\lim_{n\to\infty}P_{n}(z)$. Which we have shown to have zeros at $z=\pm q^{k}$ for $k \in \mathbb{Z}$.\\ \\

We now show that Theorem \ref{all zeroes} extends to $q$-orthogonal polynomials with strictly positive admissible weights by showing that $a_n$ and $b_n$ decay sufficiently fast for such polynomials. We start with a well known result on the zeroes of discrete orthogonal polynomials \cite[Chapter 1, Theorem 2.4]{freud2014orthogonal}.
\begin{lemma}\label{z spacing}
 The zeroes of $q$-orthogonal polynomials, with strictly positive weight $w(x)d_{q}x$, are bounded by the maximum and minimum of the $q$-lattice (i.e. between 0 and 1, or, -1 and 1 depending on the Jackson integral chosen) and do not appear more than once between lattice points.
\end{lemma}
\begin{proof}
The proof follows by contradiction. Let $\{x_{i,n}\}_{i=1}^{n}$ denote the ordered zeroes of an orthogonal polynomial, $P_{n}$, of degree $n$ (such that $x_{n,n}<x_{n-1,n}<...<x_{1,n}$). Assume that for some $i_{ext}$, $x_{i_{ext},n}$ lies outside the maximum and minimum of the $q$-lattice. It follows that 
\[ \int \left( \prod_{j=1, j\neq i_{ext}}^{n} (x-x_{j,n}) \right) \left( \prod_{i=1}^{n} (x-x_{i,n}) \right) w(x)d_{q}x > 0 \,,\]
where the polynomial on the left is degree $n-1$, while the polynomial on the right is degree $n$. This contradicts the orthogonality condition. Thus, the zeroes lie between the maximum and minimum of the $q$-lattice.\\ \\
Now consider the second part of the lemma. Assume that there exists two zeroes, $x_{i_{1},n}$ and $x_{i_{2},n}$ which lie between two consecutive points (i.e. $q^{k},q^{k-1}$) of the Jackson integral. It follows that
\[ \int \left( \prod_{j=1, j\neq i_{1},i_{2}}^{n} (x-x_{j,n}) \right) \left( \prod_{i=1}^{n} (x-x_{i,n}) \right) w(x)d_{q}x > 0 \,.\]
This is again a contradiction to orthogonality.
\end{proof}
\begin{remark}
    Note that the roots of $q$-orthogonal polynomials are real. This well known observation follows from the fact that $P_{n}$ must have real coefficients (from the recurrence relation Equation \eqref{recurrence}), combined with Lemma \ref{z spacing}. (To see this we use the same arguments as in the proof of Lemma \ref{z spacing}, omitting any complex conjugate roots if required.) It also follows that the roots must be simple (again by using the same arguments in the proof of Lemma \ref{z spacing} with the omission of any multiple roots if required).
\end{remark}

We now use Lemma \ref{z spacing} to prove that the sum of the recurrence coefficients in Equation \eqref{recurrence} is bounded.
\begin{lemma}\label{2.3}
    Given a sequence of $q$-orthogonal polynomials on $(0,1]$, with a strictly positive admissible weight, the recurrence coefficients $a_{n}$ and $b_{n}$ defined in Equation \eqref{recurrence} satisfy
    \begin{eqnarray}
        \sum_{n=0}^{\infty} b_{n} &\leq& \frac{1}{1-q} < \infty, \label{bn}\\
        \sum_{n=0}^{\infty} a_{n} &\leq& \frac{2}{(1-q)^{2}} < \infty.\label{an}
    \end{eqnarray}
\end{lemma}
\begin{proof}
Let 
\[ P_{n}(z) = z^{n} + \alpha_{n}z^{n-1} + \beta_{n}z^{n-2} + ... \,. \]
Applying Equation \eqref{recurrence} we find that
\begin{eqnarray}
    \alpha_{n} &=& \alpha_{n+1} + b_{n} , \label{c1}\\
    \beta_{n} &=& \beta_{n+1} + b_{n}\alpha_{n} + a_{n} . \label{c2}
\end{eqnarray}
By definition $a_{n}$ and $b_{n}$ are both positive (see Equation \eqref{recurrence}). We know that $-\alpha_{n}$ is equal to the sum of the roots of $P_{n}$. Furthermore, by Lemma \ref{z spacing}, $|x_{k,n}|$ is smaller then $|q^{k-1}|$. Thus, $\alpha_{n}$ is bounded by
\[ 0 < -\alpha_{n} <\frac{1}{1-q} .\]
Hence taking the telescopic sum of Equation \eqref{c1} gives
\[ \alpha_{n+1} = -\sum_{i=1}^{n}b_{i} ,\]
which implies Equation \eqref{bn}. \\ \\
Similar arguments show that Equation \eqref{c2} implies Equation \eqref{an}.
\end{proof}

\begin{remark}\label{pos remark}
    Lemma \ref{2.3} also applies to $q$-orthogonal polynomials on $[-1,0)\cup(0,1]$, with a strictly positive, symmetric, admissible weight. It follows that we can estimate an upper bound on the rate of convergence in Lemma \ref{convergence} for $q$-orthogonal polynomials with strictly positive admissible weights; $|u_{n}(z)-f(z)|\propto q^{-n}\left(2\left((1-q)z\right)^{-2}+\left( (1-q)z\right)^{-1}\right)$.
\end{remark}

By applying Lemmas \ref{z spacing} and \ref{2.3} we can deduce that Theorem \ref{all zeroes} holds for $q$-orthogonal polynomials on $(0,1]$ with strictly positive admissible weights and $q$-orthogonal polynomials on $(-1,0]\cup(0,1]$ with strictly positive, symmetric, admissible weights.\\ \\
We prove a resulting corollary which states that although the zeroes of $P_{n}$ approach $q^{k}$ as $n$ approaches infinity, they can never actually equal $q^{k}$ for any sufficiently large $n$. Note that this appears to contradict Theorem 6.1.1. in \cite{szeg1939orthogonal}, possibly due to typographical errors in the published version. 
\begin{corollary}\label{approach q but not equal}
 Let $\{P_{n}\}_{n=0}^{\infty}$ be a sequence of $q$-orthogonal polynomials with strictly positive, symmetric, admissible weight on the $q$-lattice $[-1,0)\cup(0,1]$ which satisfy Theorem \ref{all zeroes}. Consider an arbitrary point on the $q$ lattice, $q^{k}$, $k\in \mathbb{N}$. Then there exists an $N_{k}$, which is a function of $k \in \mathbb{N}$, such that the zeroes of $P_{n}$ for $n>N_{k}$ cannot equal $q^{k}$.
\end{corollary}
\begin{proof}
      Let $P_{n}(z)$ be a monic polynomial which satisfies a $q$-RHP on $[-1,0)\cup(0,1]$ with strictly positive, symmetric, admissible weight.  Assume that for a given $k$, $P_{n}(q^{k}) = 0$. As $P_{n}$ has a zero at $ q^{k}$ it is also the $n$th degree orthogonal polynomial for the weight $w(x)d_{q}x\setminus \pm q^{k}$ i.e. the Jackson integral minus the location $\pm q^{k}$. Applying Remark \ref{pos remark} we know that there exists and $N_{k}$ such that for $n>N_k$ 
      \[ \left|z^{-n}P_{n} - \prod_{j=0,j\neq k}^{\infty}(1-z^{-2}q^{2j})\right| < \left| \prod_{j=0,j\neq k}^{\infty}(1-z^{-2}q^{2j})\right| .\]
      Thus, if $n>N_{k}$ we arrive at a contradiction.
\end{proof}

\section{Conclusion}\label{conc sec}
In this paper, we provided a framework for studying properties of $q$-orthogonal polynomials as solutions of $q$-RHPs. This framework enabled us to deduce certain known properties as well as new results for $q$-orthogonal polynomials. In particular, we deduced previously known ladder operators and provided a Lax pair. Furthermore, in Theorem \ref{all zeroes} we described the asymptotic location of the zeroes of $q$-orthogonal polynomials. Although, this behaviour has been shown earlier for positive discrete weights \cite[Theorem 6.1]{szeg1939orthogonal}, we prove it under more relaxed constraints in the $q$-discrete setting.

It is interesting to note that one can determine the asymptotic location of zeroes of $q$-orthogonal polynomials using the discrete nature of the weight function for $q$-orthogonal polynomials, which manifests itself as a lack of analyticity in the RHP jump function. A similar result does not hold using the RHP for semi-classical orthogonal polynomials, particularly if the orthogonality weight is non-positive. However, the RHP for both semi-classical and $q$-orthogonal polynomials can be used to prove the form of the ladder operators and Lax pairs they respectively satisfy. 

An interesting avenue for future exploration of the $q$-RHP is to determine a form of strong asymptotics in the same vein as the previous work of Deift \textit{et al.} \cite{Deift1999strong} for semi-classical orthogonal polynomials. We note that there has been some work done in this area concerning discrete polynomials \cite{baik2007discrete}. However, these results do not seem very applicable to the $q$-orthogonal polynomials discussed in this paper. In particular, they do not accurately describe the exponentially vanishing behaviour of the recurrence coefficients. It is hoped that the $q$-RHP representation will provide a good framework to explore the strong asymptotics of $q$-orthogonal polynomials.

\enlargethispage{20pt}

\ethics{There are no ethics issues.}

\dataccess{This article has no additional data.}

\aucontribute{All the authors have contributed equally to all sections. Nalini Joshi's ORCID ID is \orcidauthorA. Tomas Lasic Latimer's ORCID ID is \orcidauthorB}

\competing{There are no competing interests.}

\funding{Nalini Joshi's research was supported by an Australian Research Council Discovery Projects \#DP200100210 and \#DP210100129. Tomas Lasic Latimer's research was supported the Australian Government Research Training Program and by the University of Sydney Postgraduate Research Supplementary Scholarship in Integrable Systems.}

\disclaimer{The ideas presented here represent our own ideas together with our understanding of those of other researchers.}


\bibliography{References}

\end{document}